\documentclass[12pt,reqno]{amsart}
\usepackage{verbatim,amscd,amsmath,amsthm,amstext,amssymb,amsfonts}
\usepackage{graphics}
\numberwithin{equation}{section}
\theoremstyle{plain}
\newtheorem{thm}{Theorem}[section]

\newtheorem{pro}{Proposition}[section]
\newtheorem{lem}{Lemma}[section]

\theoremstyle{definition}

\theoremstyle{remark}

\newtheorem*{acknow}{Acknowledgments}

\begin{document}
\title[Hypersurfaces with constant anisotropic mean curvatures]
{Hypersurfaces with constant anisotropic mean curvatures}
\author[Hui Ma]{Hui Ma}
\address{Department of Mathematical Sciences, Tsinghua University,
Beijing 100084, P.~R.~CHINA}
 \email{hma@math.tsinghua.edu.cn}
\author[Changwei Xiong]{Changwei Xiong}
\address{Department of Mathematical Sciences, Tsinghua University,
Beijing 100084, P.~R.~CHINA}
\email{xiongcw10@mails.tsinghua.edu.cn}
\thanks{This project was partially supported by NSFC grant No. 11271213 and
Tsinghua-K.U. Leuven Bilateral Scientific Cooperation Fund.}

\maketitle

\begin{abstract}
\noindent We apply the evolution method to present a new proof of the Alexandrov type theorem for constant anisotropic mean curvature
hypersurfaces in the Euclidean space $\mathbb{R}^{n+1}$.
\end{abstract}

%%%%%%%%%%%%%
\section{Introduction}
%%%%%%%%%%%%%

The classical Alexandrov theorem is one of the most remarkable results which states that
any closed embedded constant mean curvature hypersurface in the Euclidean space is a round sphere.
It has different methods to prove, for instance, Alexandrov reflection (\cite{Alexandrov}),
application of Reilly's formula (\cite{Reilly}, \cite{Ros}), Montiel-Ros' integration (\cite{MontielRos91}),
a spinorial Reilly-type inequality (\cite{HMZ}), etc.
It can also be generalized to many other ambient manifolds or
hypersurfaces with constant higher order mean curvatures
(\cite{Reilly}, \cite{Montiel}, \cite{Ros}, \cite{MontielRos91}, \cite{Hsiang2}, \cite{DM} and references therein).
Recently, S.~Brendle (\cite{Brendle}) %, \cite{BE})
proved an Alexandrov type theorem
in certain warped product manifolds, including deSitter-Schwarzschild and Reissner-Nordstorm manifolds.
His proof is based on evolution equations, which seems to have generality.

On the other hand, as a natural generalization of surfaces with constant mean curvature,
extensive research has been devoted to studying surfaces with constant anisotropic mean curvature in the Euclidean space
in the fields of analysis, geometry and material sciences (cf. \cite{Taylor}, \cite{FM}, \cite{Andrews}, \cite{Gi},
\cite{BM}, \cite{Palmer}, \cite{CM}, \cite{Winklmann}, \cite{K-Pa1}, \cite{K-Pa3}, \cite{HL1}, \cite{HL2}
and the references therein).
Let $F:\mathbb{S}^n\rightarrow \mathbb{R}^{+}$ be a smooth positive
function defined on the unit sphere which satisfies the following
convexity condition:
\begin{equation}\label{eq:convex}
A_F:=(D^2 F + F1)_x >0, \quad \forall x\in \mathbb{S}^n,
\end{equation}
where $D^2 F$ denotes the Hessian of $F$ on $\mathbb{S}^n$,
$1$ denotes the identity on $T_x\mathbb{S}^n$
and $>0$ means that the matrix is positive definite.
Now let $x:\Sigma\rightarrow \mathbb{R}^{n+1}$ be a smooth immersion of a closed orientable hypersurface and
$\nu: \Sigma\rightarrow \mathbb{S}^n$ denote its Gauss map.
Then the anisotropic surface energy of $x$ is defined as follows:
$${\mathcal F}(x)=\int_{\Sigma}F(\nu)dA.$$
Notice that if $F\equiv 1$, then ${\mathcal F}(x)$ is the usual area functional of $x$.
The algebraic $(n+1)$-volume enclosed by $\Sigma$ is given by
$$V=\frac{1}{n+1}\int_\Sigma \langle x,\nu \rangle dA.$$
It is very interesting to study the critical points of $\mathcal F$ for volume-preserving
variations.
The Euler-Lagrange equation for this constrained variational problem is
\begin{equation}\label{eq:H_F def}
H_F:=-{\rm div}_\Sigma DF +n HF = {\text constant},
\end{equation}
where $H:=-\frac{1}{n}\mathrm{tr}d\nu$ is the mean curvature of $x$.
Thus $H_F$ is called the \emph{anisotropic mean curvature} of $x$.
Notice that if $F\equiv 1$ then $H_F$ is nothing but $nH$.

Among all hypersurfaces with constant anisotropic mean curvature,
there is one class of special hypersurfaces which are the generalization of the unit spheres.
Consider the map
\begin{equation*}
\begin{split}
\varphi \colon  & \mathbb{S}^n \rightarrow {\mathbb R}^{n+1}\\
 &x \mapsto  DF_x + F(x)x,
\end{split}
\end{equation*}
where $DF$ is the gradient of $F$ on $\mathbb{S}^n$.
We call $W_F=\varphi(S^n)$ the \emph{Wulff shape} of $F$ or $\mathcal{F}$.
Under the convexity condition of $F$, $W_F$ is a smooth convex hypersurface and
$\mathcal F$ is called a \emph{parametric elliptic functional}.
When $F\equiv 1$, the Wulff shape is the unit sphere.

Observe that
$$H_F=-{\rm tr} d(\varphi\circ \nu),$$
so one can call
$$S_F:=-d(\varphi\circ \nu)=-A_F\circ d\nu$$
the \emph{anisotropic Weingarten operator} of $x$.
Let $S:=-d\nu$ be the classical Weingarten operator.
Remark that in general $S_F$ is not symmetric, but it still has real eigenvalues $\lambda_1, \cdots, \lambda_n$,
which are called \emph{anisotropic principal curvatures}.
Similar to the classical hypersurfaces theory, we have the following characterization for the
\emph{anisotropic umbilical hypersurfaces} in $\mathbb{R}^{n+1}$:

\begin{lem}[See \cite{HL1},\cite{HL2}]\label{lem:umbilical} Let $x:\Sigma^n \rightarrow {\mathbb R}^{n+1}$ be an immersed closed hypersurface.
If $\lambda_1=\lambda_2=\cdots =\lambda_n\neq 0$ holds  everywhere on $\Sigma$,
then $\Sigma$ is the Wulff shape, up to translations and hometheties.
\end{lem}

Let $\sigma_r$ be the elementary symmetric functions of the anisotropic principal curvatures
$\lambda_1,\cdots, \lambda_n$, i.e.,
$\sigma_r:=\sum_{i_1<\cdots <i_r} \lambda_{i_1}\cdots\lambda_{i_r}$ for $1\leq r \leq n$.
Set $\sigma_0=1$.
Then the $r$-th anisotropic mean curvature $H_r$ is defined by $H_r=\sigma_r/C_n^r$,
where $C_n^r=\frac{n!}{r!(n-r)!}$. In particular, $H_1=H_F/n$.

We have proved the following Alexandrov type theorem in \cite{HLMG}:

\begin{thm}\label{Thm:Alexandrov}
Let $\Sigma$ be a closed oriented hypersurface embedded in the Euclidean space $\mathbb{R}^{n+1}$.
If $H_r$ is constant for some $r=1,\cdots, n$, then $\Sigma$ is the Wulff shape, up to translations and
hometheties.
\end{thm}

In this paper, we will apply the evolution method introduced by Brendle \cite{Brendle}
to give a new proof of Theorem \ref{Thm:Alexandrov}.
In Section \ref{Sec:Pre}, we first recall hypersurfaces theory in the Euclidean space in terms of moving frames and then
we prove three fundamental equations for an immersed oriented hypersurface in $\mathbb{R}^{n+1}$
related to its anisotropic mean curvature.
In Section \ref{Sec:HK}, we use one of the fundamental equations obtained in Section \ref{Sec:Pre}
and employ the evolution method introduced by Brendle (\cite{Brendle}) to show the Heintz-Karcher type inequality
(See Theorem \ref{Thm:HK}).
In Section \ref{Sec:Proof}, we use the standard argument to prove Theorem \ref{Thm:Alexandrov}.

\begin{acknow}
The authors would like to thank Professors Haizhong~Li, Xiang~Ma, Chia-Kuei~Peng and Hui-Chun~Zhang
for their interest and helpful comments on this work.
\end{acknow}

%%%%%%%%%%%%%%
\section{Preliminaries and basic equations}
\label{Sec:Pre}
%%%%%%%%%%%%%%

For the convenience of the reader, we firstly recall the basic facts related to
anisotropic mean curvature of a hypersurface in terms of moving frames.
See more details in \cite{HL2}.

Let $x:\Sigma\rightarrow \mathbb{R}^{n+1}$ be a smooth oriented hypersurface with its Gauss map
$\nu:\Sigma\rightarrow \mathbb{S}^n$.
Let $\{E_1,\cdots, E_n\}$ be a local orthonormal frame on $\mathbb{S}^n$,
then $\{e_1:=E_1\circ\nu, \cdots, e_n:=E_n\circ\nu\}$
is a local orthonormal frame of $\Sigma$ and $e_1, \cdots, e_n, e_{n+1}=\nu$ is a local orthonormal frame on
$\mathbb{R}^{n+1}$ along $x$.
Denote the dual frames of $E_i$ and $e_i$ by $\theta_i$ and $\omega_i$, respectively,
and the corresponding connection forms by $\theta_{ij}$ and $\omega_{ij}$.

Throughout this paper, we agree on the range of indices: $1\leq i,j,\cdots \leq n$
and $1\leq A,B, \cdots \leq n+1$.
Recall that the structure equations of $y:\mathbb{S}^n\rightarrow \mathbb{R}^{n+1}$ are as follows:
\begin{equation}\label{eq:S^n}
\begin{split}
dy&=\sum_i \theta_i E_i,\quad dE_i=\sum_j \theta_{ij}E_j-\theta_i y,\\
d\theta_i&=\sum_j \theta_{ij}\wedge \theta_j, \quad
d\theta_{ij}-\sum_k \theta_{ik}\wedge \theta_{kj}=-\theta_i\wedge\theta_j,
\end{split}
\end{equation}
where $\theta_{ij}+\theta_{ji}=0$.
For a smooth positive function $F: \mathbb{S}^n\rightarrow \mathbb{R}^{+}$, we define the covariant
differentials $F_i$, $F_{ij}$ and $F_{ijk}$ as follows
\begin{equation}\label{eq:F_i}
\begin{split}
&dF=\sum_i F_i\theta_i, \\
&dF_i+\sum_j F_j \theta_{ji}=\sum_j F_{ij}\theta_j, \\%\quad D^2F=F_{ij}\theta_i\otimes \theta_j,\\
&dF_{ij}+F_{ik}\theta_{kj}+F_{kj}\theta_{ki}=\sum_k F_{ijk}\theta_k.
\end{split}
\end{equation}
It follows from \eqref{eq:S^n} and Ricci identity that
$$F_{ijk}-F_{ikj}=F_j \delta_{ik}-F_k \delta_{ij},$$
which implies that $(F_{ij}+F\delta_{ij})_{,k}=(F_{ik}+F\delta_{ik})_{,j}$.
Denote the coefficients of $A_F$ by $A_{ij}=F_{ij}+F\delta_{ij}$, then we have
\begin{equation}\label{eq:Aijk=Aikj}
A_{ij,k}=A_{ik,j},
\end{equation}
where
\begin{equation}\label{eq:dA_ij}
\sum_k A_{ij,k}\theta_k=dA_{ij}+A_{ik}\theta_{kj}+A_{kj}\theta_{ki}.
\end{equation}

\medskip

The structure equations of $x: \Sigma\rightarrow \mathbb{R}^{n+1}$ are given by
\begin{equation}\label{eq:Sigma}
\begin{split}
dx&=\sum_i \omega_i e_i, \\
de_i&=\sum_j \omega_{ij}e_j+\sum_j h_{ij}\omega_j \nu, \quad d\nu=-\sum_{ij} h_{ij}\omega_j e_i,\\
d\omega_i&=\sum_j \omega_{ij}\wedge \omega_j, \quad
d\omega_{ij}=\omega_{ik}\wedge\omega_{kj}-\frac{1}{2}R_{ijkl}\omega_k\wedge\omega_l,
\end{split}
\end{equation}
where $\omega_{i n+1}=\sum_j h_{ij}\omega_j$ and $h_{ij}=h_{ji}$.
Making use of \eqref{eq:S^n}, we get
$$de_i=d(E_i\circ\nu)=\nu^*dE_i=\sum_j \nu^*\theta_{ij}e_j-\nu^*\theta_i\nu,$$
thus we have
\begin{equation}\label{eq:omega theta}
\omega_{ij}=\nu^* \theta_{ij}, \quad\quad \nu^*\theta_i=-\sum_j h_{ij}\omega_j.
\end{equation}

Let $f$ be a smooth function on $\Sigma$. Define the first, second covariant derivatives $f_i$, $f_{ij}$ as follows
$$df=\sum_i f_i \omega_i, \quad\quad\quad df_i+\sum_j f_j \omega_{ji}=\sum_j f_{ij}\omega_{j}.$$
Considering $x$ and $\nu$ as smooth functions on $\Sigma$, we have
\begin{eqnarray}
&&x_i=e_i,\quad x_{ij}=h_{ij}\nu,\label{eq:x_i x_ij}\\
&&\nu_i=-h_{ij}e_j, \quad
\nu_{ij}=-h_{ikj}e_k-h_{ik}h_{kj}\nu,\label{eq:nu_i nu_ij}
\end{eqnarray}
where the covariant derivative $h_{ijk}$ is defined by
$dh_{ij}+h_{kj}\omega_{ki}+h_{ik}\omega_{kj}=\sum_k h_{ijk}\omega_k$.

For a smooth positive function $F:\mathbb{S}^{n}\rightarrow \mathbb{R}^{+}$,
$F\circ \nu$ is a function on $\Sigma$. We define the covariant derivatives $(F\circ\nu)_i$, $(F_i\circ\nu)_j$ and
$(A_{ij}\circ\nu)_k$ by
\begin{equation}\label{eq:Fnui}
\begin{split}
&d(F\circ\nu)=\sum_i (F\circ\nu)_i\omega_i,\\
&d(F_i\circ\nu)+\sum_j (F_j\circ\nu)\omega_{ji}=\sum_j (F_i\circ\nu)_j\omega_j,\\
&d(A_{ij}\circ\nu)+\sum_k (A_{kj}\circ\nu)\omega_{ki}+\sum_k(A_{ik}\circ\nu)\omega_{kj}
=\sum_k (A_{ij}\circ\nu)_k\omega_k.
\end{split}
\end{equation}
Taking $\nu^*$ on both sides of equations \eqref{eq:F_i} and using \eqref{eq:omega theta} and \eqref{eq:Fnui}, we get
 (\cite{HL2})
\begin{eqnarray}
&&(F\circ\nu)_i=-\sum_j h_{ij}(F_j\circ\nu),\label{eq:Fnu_i}\\
&&(F_i\circ\nu)_j=-h_{jk}(F_{ik}\circ\nu),\label{eq:Fnu_ij}\\
&&(A_{ij}\circ\nu)_k=-\sum_{p}(A_{ijp}\circ\nu)h_{pk}.\label{eq:A_ijk}
\end{eqnarray}
Denote $S_F e_i=\sum_j s_{ij}e_j$, then
\begin{equation}
s_{ij}=\sum_l (A_{il}\circ\nu )h_{lj}.
\end{equation}
Thus we have
\begin{equation}\label{eq:s_ijk}
s_{ijk}=-\sum_{lp} (A_{ilp}\circ \nu)h_{pk} h_{lj}+\sum_l(A_{il}\circ\nu)h_{ljk},
\end{equation}
and
\begin{equation}\label{eq:Codazzi}
s_{ijk}=s_{ikj}.
\end{equation}

Denote by $p:=\langle x,\nu\rangle$ the support function. % and we abbreviate $F=F\circ\nu$.
The following identities were already derived in \cite{Winklmann} in the case when
the anisotropic mean curvature is constant and \eqref{eq:Delta_F nu} has been obtained in \cite{CM}.
\eqref{eq:Delta_F F} will play an important role in the proof of Theorem \ref{Thm:HK}.

\begin{pro}
\begin{eqnarray}
&&\Delta_F (F\circ\nu)+\mathrm{tr}(A_F S^2)F+\langle \nabla H_F, DF|_{\nu}\rangle =\mathrm{tr}(S_F^2)\label{eq:Delta_F F}\\
&&\Delta_F\nu+\mathrm{tr}(A_F S^2)\nu+\nabla H_F=0,\label{eq:Delta_F nu}\\
&&\Delta_F p+\mathrm{tr}(A_F S^2)p+H_F+\langle x, \nabla H_F\rangle=0,\label{eq:Delta_F p}
\end{eqnarray}
where $\Delta_F f:=\mathrm{div}(A_F \nabla f)$, $\nabla f$ denotes the gradient of $f$ with respect to the induced
metric on $\Sigma$ and $\mathrm{tr}(S_F^2)=\sum_{i,j} s_{ij}s_{ji}$ and $\mathrm{tr}(S^2)=\sum_{i,j}h_{ij}^2$.
\end{pro}

\begin{proof}Making use of \eqref{eq:A_ijk}, \eqref{eq:Fnu_i}, \eqref{eq:Fnu_ij}, \eqref{eq:s_ijk} and \eqref{eq:Codazzi},
we get
\begin{eqnarray*}
\Delta_F (F\circ\nu)&=&\sum_{k,i}((A_{ik}\circ \nu)(F\circ \nu)_k)_{i}\\
&=&\sum_{j,i,k,p}(A_{ikp}\circ\nu)h_{pi}h_{kj} (F_j\circ\nu)-\sum_{k,i,j}(A_{ik}\circ\nu)h_{kji}(F_j\circ\nu)
-\sum_{k,i,j}(A_{ik}\circ\nu)h_{kj}(F_j\circ\nu)_i\\
&=&\sum_j(\sum_{i,k,p}(A_{ikp}\circ\nu)h_{pi}h_{kj}-\sum_{i,k}(A_{ik}\circ\nu)h_{kji})(F_j\circ\nu)
+\sum_{k,i,j,p}(A_{ik}\circ\nu)h_{kj}h_{ip}(F_{jp}\circ\nu)\\
&=&-\sum_j s_{iji}(F_j\circ\nu)+\sum_{k,i,j,p}(A_{ik}\circ\nu)h_{kj}h_{ip}(A_{jp}\circ\nu-(F\circ\nu)\delta_{jp})\\
&=&-\sum_j s_{iij}(F_j\circ\nu)+\sum_{i,j}s_{ij} s_{ji}- \sum_{i,j,k}(A_{ik}\circ\nu)h_{kj}h_{ij}(F\circ\nu)\\
&=&-\sum_j (H_F)_j (F_j\circ\nu) +\mathrm{tr}(S_F^2)-\mathrm{tr}(A_F S^2)F\circ\nu.
\end{eqnarray*}
This proves \eqref{eq:Delta_F F}.

By \eqref{eq:A_ijk}, \eqref{eq:nu_i nu_ij}, \eqref{eq:s_ijk} and direct calculations, we get
\begin{eqnarray*}
\Delta_F\nu&=&\sum_{k,i}((A_{ik}\circ\nu)\nu_k)_i\\
&=&\sum_{i,k,p,l}(A_{ikp}\circ\nu)h_{pi}h_{kl}e_l-\sum_{i,k,p}(A_{ik}\circ\nu)(h_{kip}e_p+h_{ip}h_{pk}\nu)\\
&=&-\sum_{l} s_{iil}e_l-\sum_{i,k,p}(A_{ik}\circ\nu)h_{kp}h_{pi}\nu\\
&=&-\sum_l (H_F)_l e_l-\mathrm{tr}(A_F S^2)\nu,
\end{eqnarray*}
which immediately verifies \eqref{eq:Delta_F nu}.

It follows from $dp=\langle x,d\nu\rangle=\langle x, \nu_i\omega_i\rangle$ that
$p_i=\langle x,\nu_i\rangle$. Using \eqref{eq:x_i x_ij} and \eqref{eq:nu_i nu_ij}, we get
\begin{eqnarray*}
p_{ij}&=&\langle x,\nu_i\rangle_j=\langle e_j, -h_{ik}e_k\rangle +\langle x, -h_{ijk}e_k-h_{ik}h_{kj}\nu\rangle,\\
&=&-h_{ij}-h_{ijk}\langle x,e_k\rangle -h_{ik}h_{kj}p.
\end{eqnarray*}
Together with \eqref{eq:A_ijk} and \eqref{eq:s_ijk}, it yields
\begin{eqnarray*}
\Delta_F p&=&\sum_{i,k}((A_{ik}\circ\nu)p_k)_i\\
&=&\sum_{i,k,q,l}(A_{ikq}\circ\nu)h_{qi}h_{kl}\langle x,e_l\rangle
-\sum_{i,k,q}(A_{ik}\circ\nu)(h_{ki}+h_{kiq}\langle x, e_q\rangle +h_{kq}h_{qi}p),\\
&=&-\sum_l s_{iil}\langle x,e_l\rangle -H_F-\mathrm{tr}(A_F S^2)p\\
&=&-\langle x, \nabla H_F\rangle -H_F-\mathrm{tr}(A_FS^2)p,
\end{eqnarray*}
which completes the proof of \eqref{eq:Delta_F p}.
\end{proof}

The following Minkowski formula and its higher order version were obtained in \cite{HL1} and \cite{HL2}.

\begin{pro}
Let $\Sigma$ be a closed orientable hypersurface immersed in $\mathbb{R}^{n+1}$. Then
\begin{equation*}
\int_{\Sigma}(nF\circ\nu +H_F\langle x,\nu \rangle)dA=0,
\end{equation*}
More generally,
\begin{equation}\label{eq:Minkowski_r}
\int_{\Sigma}(H_r F\circ\nu+H_{r+1}\langle x,\nu\rangle)dA=0
\end{equation}
for $r=0, 1, \cdots, n-1$.
\end{pro}

%%%%%%%%%%%%%%%%%%%%
\section{Heintze-Karcher type inequality}
\label{Sec:HK}
%%%%%%%%%%%%%%%%%%%%

Consider a closed orientable hypersurface $\Sigma$ embedded in $\mathbb{R}^{n+1}$.
Denote by $\nu$ the inner unit normal vector field to $\Sigma$.
Assume that the anisotropic mean curvature $H_F$ with respect to the inner normal $\nu$ is everywhere positive
on $\Sigma$.
Suppose that there exists a domain $\Omega\subset \mathbb{R}^{n+1}$ such that $\partial\Omega=\Sigma$.

Given a smooth positive $F$ with convexity condition,
we can associate the dual norm $F^*: \mathbb{R}^{n+1}\rightarrow \mathbb{R}$ defined by (\cite{Morgan}, \cite{HLMG})
$$F^*(x)=\mathrm{sup}\{\frac{\langle x,z\rangle}{F(z)}|z\in\mathbb{S}^n\}.$$
Then we can define the $F$-distance function $d_F:\mathbb{R}^{n+1}\times \mathbb{R}^{n+1}\rightarrow \mathbb{R}$
to be $d_F(x,y)=F^*(y-x)$. Note that in general $d_F(x,y)\neq d_F(y,x)$ and
when $F\equiv 1$, $d_F$ is the Euclidean distance function $d$.

For $P\in \mathbb{R}^{n+1}$, let $\rho_F(P)=d_{F}(P,\Sigma)$.
Then we can foliate $\Omega$ by a smooth family of hypersurfaces
$$\Sigma_t:=\{\Phi(x,t)=x+t\nu_F(x)|x\in \Sigma \textrm{ and } t<c(x)\},$$
where $c:\Sigma\rightarrow \mathbb{R}\cup \{+\infty\}$ is the $F$-cut function of $\Sigma$ (See \cite{HLMG}).
The point $\Phi(x,t)$ on $\Sigma_t$ satisfies
\begin{equation}\label{eq:flow}
\frac{\partial}{\partial t}\Phi(x,t)=\nu_F=:f\nu+\xi,
\end{equation}
and $\Sigma_t$ will disappear at the time $T=max_{\Sigma}c(x)$,
where $f$ is a smooth function defined on $\Sigma\times I\subset \Sigma \times \mathbb{R}$
and $\xi$ is tangent to $\Sigma_t$.
Remark that the anisotropic normal $\nu_F=\phi\circ\nu=F(\nu)\nu+DF|_\nu$.
So $f=F(\nu)$ and $\xi=DF|_{\nu}$ in \eqref{eq:flow}.

\begin{pro}\label{pro:flow}
Under the flow \eqref{eq:flow}, we have the following evolution equations:
\begin{eqnarray}
&&\frac{\partial }{\partial t}dA_t=(\mathrm{div}\xi-nHf)dA_t,\\
&&\frac{\partial \nu}{\partial t}=-\nabla f+d\nu(\xi),\\
%&&\frac{\partial H}{\partial t}=\Delta_{\Sigma_t} f+n\langle \nabla H,\xi\rangle +\mathrm{tr}(S^2)f,\\
&&\frac{\partial}{\partial t}F\circ \nu_t=\langle D F(\nu_t), -\nabla f+ d\nu_t(\xi)\rangle,\\
&&\frac{\partial}{\partial t}H_F=\Delta_F f+\mathrm{tr}(A_F S^2)f+\langle \nabla H_F, \xi \rangle.
\end{eqnarray}
\end{pro}
\begin{proof}
In fact, the first three equations are classical results and the last one follows from Lemma 2.1 of \cite{HL1}
or (4.20) in \cite{CM}.
%or other papers involving the computations of the second variations of $\mathcal{F}(x)$.
\end{proof}

Define
\begin{equation*}
Q(t):=n\int_{\Sigma_t}\frac{F\circ\nu_t}{H_F} dA_t.
\end{equation*}
From the identities in Proposition \ref{pro:flow}, we have
\begin{eqnarray*}
\frac{1}{n}Q^{\prime}(t)&=&\int_{\Sigma_t}\frac{\partial(F\circ\nu_t)}{\partial t}\frac{1}{H_F}dA_t
+(F\circ \nu_t)\frac{\partial}{\partial t}(\frac{1}{H_F})dA_t
+\frac{F\circ\nu_t}{H_F}\frac{\partial}{\partial t}dA_t\\
&=&\int_{\Sigma_t} \{f\mathrm{div}(\frac{1}{H_F}DF|_{\nu_t})
 - \mathrm{div}(\frac{f}{H_F}DF|_{\nu_t})+\mathrm{div}(\frac{F\circ\nu_t}{H_F}\xi)\\
&\quad& \quad -\frac{f}{H_F}nH(F\circ \nu_t) -\frac{F\circ \nu_t}{H_F^2}(\Delta_F f+\mathrm{tr}(A_F S^2)f)\} dA_t.
\end{eqnarray*}
Thus by the divergence theorem and the definition of $H_F$ \eqref{eq:H_F def} we get
\begin{eqnarray*}
\frac{1}{n}Q^\prime(t)&=& \int_{\Sigma_t} \{f\langle \nabla \frac{1}{H_F}, DF\circ \nu_t\rangle
+\frac{f}{H_F}[\mathrm{div}(DF\circ\nu_t)-nH F\circ\nu_t]\\
&\quad & \quad -\frac{F\circ\nu_t}{H_F^2}[\Delta_F f+\mathrm{tr}(A_F S^2)f]\}dA_t \\
&=&\int_{\Sigma_t}\{-\frac{f}{H_F^2}\langle \nabla H_F,DF\circ\nu_t \rangle-f
-\frac{F\circ \nu_t}{H_F^2}[\Delta_F f+\mathrm{tr}(A_F S^2)f]\}dA_t.
\end{eqnarray*}
Now taking account into $f=F\circ\nu_t$ and \eqref{eq:Delta_F F}, we get
\begin{eqnarray*}
\frac{1}{n}Q^\prime(t)&=&\int_{\Sigma_t}\{-\frac{F\circ \nu_t}{H_F^2}[\langle \nabla H_F,DF\circ\nu_t \rangle
+ \Delta_F (F\circ\nu_t)+\mathrm{tr}(A_F S^2)F\circ\nu_t]-F\circ \nu_t\}dA_t.\\
&=&-\int_{\Sigma_t}(\frac{\mathrm{tr}(S_F^2)}{H_F^2}+1)F\circ\nu_t dA_t\\
&\leq & -(1+\frac{1}{n})\int_{\Sigma_t}F\circ\nu_t dA_t <0,
\end{eqnarray*}
where we have used $\frac{\mathrm{tr}(S_F^2)}{H_F^2}\geq \frac{1}{n}$ and the equal sign holds if and only if
$S_F=\frac{H_F}{n}\mathrm{Id}$.
This shows that $Q(t)$ is monotone decreasing.

For $0<\tau <T$,
\begin{eqnarray*}
Q(0)-Q(\tau)&=&-\int_0^\tau Q^{\prime}(t) dt \geq (n+1)\int_0^\tau \int_{\Sigma_t} F\circ\nu_t  dAdt\\
&=&(n+1)\int_{\Omega\cap \{\rho_F\leq \tau\}} dV,
\end{eqnarray*}
where the last equality follows from the co-area formula.
Let $\tau\rightarrow T$. Then we obtain the following Heintze-Karcher type integral inequality
that one can find also in \cite{HLMG} where it was proved using the ideas of \cite{MontielRos91}.

\begin{thm}\label{Thm:HK}
 Let $x:\Sigma\rightarrow \mathbb{R}^{n+1}$ be a closed hypersurface embedded into the Euclidean space.
If the anisotropic mean curvature $H_F$ with respect to the inner normal $\nu$ is everywhere positive on $\Sigma$,
then we have
\begin{equation}\label{eq:HK}
n\int_{\Sigma}\frac{F\circ \nu}{H_F} dA \geq (n+1)V(\Omega),
\end{equation}
where $V(\Omega)$ is the volume of the compact domain $\Omega$ determined by $\Sigma$.
Moreover, the equality holds if and only if $\Sigma$ is anisotropic umbilical.
\end{thm}

%%%%%%%%%%%%%%%%%%%
\section{Proof of the main theorem}
\label{Sec:Proof}
%%%%%%%%%%%%%%%%%%%

Once we have Minkowski formula \eqref{eq:Minkowski_r} and Heintze-Karcher type inequality \eqref{eq:HK},
it is straightforward to prove the Alexandrov type theorem by the standard argument (\cite{MontielRos91}, \cite{HLMG}):

Since the hypersurface $\Sigma$ is compact and oriented in $\mathbb{R}^{n+1}$,
there is a point on $\Sigma$ where all the principal curvatures are positive with respect to the unit inner normal $\nu$.
It follows from $A_F$ is positive definite that all the anisotropic principal curvatures at this point are positive.
So the $r$-th anisotropic mean curvature is a positive constant.
It follows from the G{\aa}rding inequality (c.f. Lemma 1 in \cite{MontielRos91}) that
$H_{r-1}\geq H_r^{\frac{r-1}{r}}$ and
$H_r^{1/r}\leq H_1=H_F/n$ for $r=1, \cdots, n-1$ and the equality happens only at anisotropic points if $r\geq 2$.
Thus the Heintze-Karcher type inequality implies
\begin{equation}\label{eq:HK_r}
(n+1) H_r^{1/r} {\rm V}(\Omega)\leq \int_{\Sigma} F\circ\nu dA
\end{equation}
and the equality holds if and only if $\Sigma$ is anisotropic umbilical.
Combined with the Minkowski formula \eqref{eq:Minkowski_r},
\begin{equation*}
\begin{split}
0&=\int_{\Sigma}(H_{r-1}F\circ\nu+H_r\langle x,\nu\rangle)dA
\geq \int_{\Sigma}(H_r^{\frac{r-1}{r}}F\circ\nu +H_r\langle x,\nu\rangle)dA\\
&=H_r^{\frac{r-1}{r}}\int_{\Sigma}(F\circ\nu +H_r^{1/r}\langle x,\nu\rangle)dA.
\end{split}
\end{equation*}
Since $H_r$ is a positive constant, we have
$$(n+1) H_r^{1/r} {\rm V}(\Omega)=-H_r^{1/r}\int_{\Sigma}\langle x,\nu\rangle dA \geq \int_{\Sigma} F\circ\nu dA.$$
Hence the equal sign in \eqref{eq:HK_r} is attached. Together with Lemma \ref{lem:umbilical},
the proof is complete.

%%%%%%%%%%%%%%%%%%%%%%%%%%%%%%%%%%%%%%%%%%%%%%%%%%%%%%%%%%%%%%%%%%%%%%

\end{document}